\documentclass[a4paper,12pt]{amsart}

\usepackage[colorlinks,citecolor = red, linkcolor=blue,hyperindex]{hyperref}
\usepackage{euscript,eufrak,verbatim}
\usepackage[psamsfonts]{amssymb}
\usepackage[usenames]{color}
\usepackage[curve]{xypic}
\usepackage{graphicx}
\usepackage{mathrsfs}
\usepackage{calligra,amsmath,amssymb}

\usepackage{float}
\usepackage{bbm}
 \usepackage{color}
\newtheorem{thm}{Theorem}[section]
\newtheorem{lem}[thm]{Lemma}

\newtheorem{prop}[thm]{Proposition}
\newtheorem{cor}[thm]{Corollary}
\newtheorem{defn}[thm]{Definition}

\newtheorem*{question*}{Question}
\newtheorem*{prob*}{Problem}

\def\R{\mathbb{R}}
\def\N{\mathbb{N}}
\def\Z{\mathbb{Z}}
\def\Q{\mathbb{Q}}
\def\E{\mathbb{E}}

\def\FF{F}
\def\O{\mathcal{O}}

\def\X{\mathcal{X}}

\def\C{\mathbb{C}}

\def\PP{\mathbb{P}}
\def\1{\mathbbm{1}}

\def\Var {\mathrm{Var}}
\def\Cov {\mathrm{Cov}}

\def \spann {\mathrm{span}}

\def\Xint#1{\mathchoice
{\XXint\displaystyle\textstyle{#1}}%
{\XXint\textstyle\scriptstyle{#1}}%
{\XXint\scriptstyle\scriptscriptstyle{#1}}%
{\XXint\scriptscriptstyle\scriptscriptstyle{#1}}%
\!\int}
\def\XXint#1#2#3{{\setbox0=\hbox{$#1{#2#3}{\int}$ }
\vcenter{\hbox{$#2#3$ }}\kern-.6\wd0}}

\def\dashint{\Xint-}

\begin{document}

\title[Stationary  DPP  in non-Archimedean fields]{Rigid stationary determinantal processes in non-Archimedean fields}


\author
{Yanqi Qiu}

\address
{Yanqi Qiu: CNRS, Institut de Math{\'e}matiques de Toulouse, Universit{\'e} Paul Sabatier, 118, Route de Narbonne, F-31062 Toulouse Cedex 9}

\email{yqi.qiu@gmail.com}
%
%

\begin{abstract}
 Let $F$ be a non-discrete non-Archimedean local field. For any subset $S\subset F$ with finite Haar measure, there is a stationary determinantal point process on $F$ with correlation kernel $\widehat{\mathbbm{1}}_S(x-y)$, where $\widehat{\mathbbm{1}}_S$ is the Fourier transform of the indicator function $\mathbbm{1}_S$. In this note, we give a geometrical condition on the subset $S$, such that the associated determinantal point process is rigid in the sense of Ghosh and Peres. Our geometrical condition is very different from the Euclidean case.
\end{abstract}


\maketitle

\section{Main result}
Let $F$ be a  non-discrete non-Archimedean local field. Write $\O_F$ for the  its valuation ring with maximal ideal $\mathcal{M}_F$. Let $q: = p^e$ be the number of elements of the finite residue field $\O_F/\mathcal{M}_F$, where $p$ is a prime number and $e \in \N$.  Fix the standard norm $| \cdot |$ on $F$.  Let $\mathfrak{m}$ be  the  Haar measure on $F$ normalized such that $\mathfrak{m}(\O_F) = 1$. 

 By a random point process  $\mathcal{X}$ on $F$, we mean a random {\it locally finite} subset of $F$. The main objects under consideration in this note are stationary determinantal point processes on $F$.  For the background on determinantal point processes, the reader is referred to \cite{DPP-HKPV, DPP-L, DPP-S, DPP-M} and references therein.

\begin{defn}[Ghosh and Peres \cite{Ghosh-sine, Ghosh-rigid}]\label{def-rigid}
A random point process $\mathcal{X}$ on $F$ is number rigid, if for any bounded open subset $B\subset F$, the number $\#(\mathcal{X} \cap B)$ of particles of the random point process $\mathcal{X}$ inside $B$, is almost surely determined by $\mathcal{X}\setminus B$. 
\end{defn}
We refer to \cite{Ghosh-rigid3, QB3, Buf-rigid, QB4, Osada-Shirai} for further references on the number-rigidity property for determinantal point processes. 

 Let $S \subset F$ be a measurable subset such that $0< \mathfrak{m}(S) < \infty$. By Macch\`i-Soshnikov theorem, we may introduce a determinantal point process on $F$, denoted by $\mathcal{X}_S$, whose correlation kernel is given by
\[
K_S(x,y)  = \widehat{\1_S}(x-y),  \quad ( x, y \in F),
\]
where $\widehat{\1_S}$ is the Fourier transform of the indicator function $\1_S$ of the set $S$, see \S \ref{sec-F} for the precise definition of the Fourier transform in the non-Archimedean setting. The translation-invariance of the kernel $K_S(x, y)$ implies that the random point process $\mathcal{X}_S$ is stationary, that is,  the probability distribution of $\X_S$ is invariant under translations.

Our main result is 
\begin{thm}\label{main-thm}
Assume that $S \subset F$ is a measurable subset  such that $0< \mathfrak{m}(S) < \infty$ and
\begin{align}\label{con-geo}
 \dashint_{B(0, q^{-n})}\mathfrak{m}(S\setminus (S+y) ) \mathrm{d}\mathfrak{m}(y) = O (q^{-n}) \text{\, as $n \to \infty$,}
\end{align}
where $ \dashint_{B(0, q^{-n})} \mathrm{d}\mathfrak{m} $ is the normalized integration. Then the determinantal point process induced by the kernel $\widehat{\1_S}(x-y)$ is number rigid. 
\end{thm}

The geometrical condition \eqref{con-geo} can be replaced by an analytic condition on the $L^2$-decay of the Fourier transform $\widehat{\1_S}$. 

\begin{thm}\label{main-thm-bis}
Assume that $S \subset F$ is a measurable subset  such that $0< \mathfrak{m}(S) < \infty$ and
\begin{align}\label{con-anl}
 \int_{F\setminus B(0, q^{n})}    |\widehat{\1_S}|^2 \mathrm{d}\mathfrak{m} = O (q^{-n}) \text{\, as $n \to \infty$.}
\end{align}
Then the determinantal point process induced by the kernel $\widehat{\1_S}(x-y)$ is number rigid. 
\end{thm}

We note that in Euclidean case, the {\bf stationary} determinantal point processes  that are known to be  number rigid are
\begin{itemize}
\item  \cite{Ghosh-sine}  the Dyson sine process or slightly more generally, the determinantal point processes on $\R$ with a correlation kernel 
\[
\widehat{\1_S} (x -y) , \quad (x, y \in \R), 
\]
where $S$ is a {\bf finite union of intervals}; 
\item  \cite{Ghosh-rigid} the Ginibre point process.  
\end{itemize}
But the Ginibre point process is induced by the kernel
\[
e^{-\frac{|z|^2 +|w|^2}{2} + z\bar{w}}, \quad (z, w\in \C), 
\]
which is not given by the Fourier transform of any function on $\C\simeq \R^2$. Therefore,  in Euclidean case, only in dimension one do we know examples of {\bf stationary} determinantal point processes whose correlation kernels are convolution kernels. It has also been mentioned in \cite{QB4} that the existing methods do not seem to produce number rigid stationary determinantal point processes  in $\R^d$ with $d \ge 2$.  

While in non-Archimedean setting,  there exist trivial examples of stationary determinantal point processes: for example, in the case of $p$-adic number field $\Q_p$, for any $n \in \Z$,  by the identity 
\[
p^{-n}\1_{p^{-n} \Z_p}   = \widehat{\1_{p^n \Z_p}}, 
\] 
we know that the kernel 
\begin{align}\label{trivial-kernel}
p^{-n}\1_{p^{-n} \Z_p} (x -y)  = \widehat{\1_{p^n \Z_p}}(x-y), \quad (x, y \in \Q_p)
\end{align}
induces a determinantal point process on $\Q_p$.  This determinantal point process can be trivially verified to be number rigid since  for any pair $x, y \in \Q_p$ such that $|x -y|_p > p^n$, we have
\[
\widehat{\1_{p^n \Z_p}}(x-y) = 0
\]
and this implies that the determinantal point process induced by the kernel \eqref{trivial-kernel} is the union of countably many independent copies of determinantal point processes on  the translates of $p^{-n}\Z_p$, each of them has exactly one particle. By the same reason, if $S\subset F$ is a {\bf union of finitely many balls}, then the determinantal point process induced by the kernel $\widehat{\1_S} (x -y)$ is trivially number rigid. Of course, in these trivial examples, the sets $S$ satisfy trivially the condition \eqref{con-geo}.  

Theorem \ref{main-thm} produces non-trivial examples of Borel subsets $S\subset F$ such that the associated stationary determinantal point processes are number rigid, see \S \ref{sec-ex} for such examples. Since  any finite dimensional vector space over $F$ can be seen as a finite extension of the field $F$, which is again  a non-Archimedean local field, Theorem \ref{main-thm} produces non-trivial examples of stationary determinantal point processes on any finite dimensional vector space over $F$. 

Recall that any open set of the real line $\R$ is a countable union of intervals,  inspired by our result in non-Archimedean situation,  it is natural to ask the following 
\begin{question*}
Let $S \subset \R$ be an {\bf open} subset with finite positive Lebesgue measure. Is the determinantal point process on $\R$ induced by the correlation kernel 
\[
\widehat{\1_S} (x- y), \quad( x, y \in \R)
\]
number rigid? 
\end{question*}

\section{Preliminaries}
\subsection{Notation}\label{sec-F}

We recall some basic notion in the theory of local fields.  Let $F$ be a non-discrete  non-Archimedean local field.  
The classification of local fields (see  Ramakrishnan and Valenza\cite[Theorem 4-12]{DR-fourier}) implies that $F$ is isomorphic to one of the following fields: 
\begin{itemize}
\item a finite extension of the field $\Q_p$ of $p$-adic numbers for some prime $p$.
\item the field of formal Laurent series  over a finite field.
\end{itemize}

Let $| \cdot |$ denote the  {\em absolute value} on $F$ and let $d(\cdot, \cdot)$ denote the metric on $F$ defined by $d(x, y) = | x-y|$.  The set $\O_F= \{x\in F: |x| \le 1\}$ forms a subring of $F$ and is called the ring of integers or the valuation ring of $F$.  The subset $\mathcal{M}_F =\{x\in F: | x| < 1\}$
 is the unique maximal  ideal of the integer-ring $\O_F$.  The quotient  $\O_F/\mathcal{M}_F$ is a  finite field with cardinality 
 \[
 \# (\O_F/ \mathcal{M}_F) = q = p^e, \quad \text{$p$ is prime and $e \in\N$}. 
 \] 
By fixing any  element $\varpi \in F$ with $|\varpi|= q^{-1}$, we have  
\[
\mathcal{M}_F= \varpi \O_F. 
\]

Denote by $\widehat{\FF}$ the Pontryagin dual of the additive group $\FF$. Elements in $\widehat{F}$ are  called characters of $F$.  Throughout the note,  we fix a non-trivial   character $\chi \in \widehat{\FF}$ such that 
\begin{align}\label{ass-chi}  
\text{  $\chi|_{\O_F} \equiv 1 $ and $\chi$ is  not constant on  $\varpi^{-1} \O_F$.}
\end{align}
For any $y \in \FF$, define a character   $\chi_y \in \widehat{\FF}$ by
$\chi_y(x) =\chi(yx)$, then  the map $y\mapsto \chi_{y}$ defines a group isomorphism from $\FF$ to $\widehat{\FF}$.

Let $\mathfrak{m}$ be  the  Haar measure on the additive group $F$ normalized such that $\mathfrak{m}(\O_F) = 1$. Given any function $f\in L^1(F, \mathfrak{m})$, its {\em Fourier transform}  is  defined by
\[
\widehat{f}(y) = \int_{F} f(x) \chi(xy) \mathrm{d}\mathfrak{m} (x),  \quad y \in F.
\]

\subsection{Determinantal point processes on $F$}
We say a random point process $\mathcal{X}$ is a determinantal point process induced by a correlation kernel $K: F \times F \rightarrow \C$, if for any positive integer $n$ and for any compactly supported bounded measurable function $\varphi: F^n \rightarrow \C$, we have 
\begin{align}\label{def-DPP}
\begin{split}
& \E \Big( \sum_{x_1, \dots, x_n \in \X}^*  \varphi(x_1, \dots, x_n) \Big) 
\\
=&   \int_{F^k} \varphi(y_1, \dots, y_n) \det( K(y_i, y_j))_{1 \le i, j \le n} \mathrm{d}\mathfrak{m} (y_1) \dots \mathrm{d}\mathfrak{m} (y_n),
\end{split}
\end{align}
where $\sum\limits^{*}$ denotes the sum over all ordered $n$-tuples of {\it distinct} points $(x_1, \dots, x_n) \in \mathcal{X}^n$.

\section{Kolmogorov minimality}

Let  $\Gamma = F/\O_F$  the quotient additive group. Then $\Gamma$ is a discrete countable group. 
Elements in $\Gamma$ will be denoted either by $\gamma \in \Gamma$ or $[x]\in \Gamma$ with $x \in F$.   We equip $\Gamma$ with an absolute value, denoted again by $| \cdot |$ and defined by 
\[
|[x]|: = \left\{\begin{array}{cc} |x|,  & x\notin \O_F \vspace{2mm} \\ 0, & x\in \O_F \end{array}\right.. 
\]
Note that the group $\Gamma$ is identified naturally with the Pontryagin dual of the additive group $\O_F$ by the following well-defined pairing 
\begin{align*}
\begin{array}{ccc}
\Gamma   \times \O_F&\rightarrow & \C
\\
([x], h) & \mapsto & \chi(xh)
\end{array},
\end{align*}
where $\chi$ is the fixed character of $F$ satisfying \eqref{ass-chi}. 

A $\Gamma$-indexed  family  $(X_\gamma)_{\gamma \in \Gamma}$ of $\C$-valued random variables defined on a common probability space $(\Omega, \PP)$ is called a $\Gamma$-indexed stochastic process. It is called (weakly)  $\Gamma$-stationary, if $X_\gamma \in L^2(\Omega, \PP)$ for all $\gamma\in \Gamma$ and for any $\gamma, \gamma',  \alpha \in \Gamma$, we have 
\begin{align*}
\E(X_\gamma) =  \E(X_{\gamma'}) \text{\, and \,} \Cov(X_\gamma, X_{\gamma'}) = \Cov(X_{\gamma + \alpha}, X_{\gamma' + \alpha}), 
\end{align*}
where $\E(X_\gamma)$ denotes the expectation of  $X_\gamma$ and $\Cov(X_\gamma, X_{\gamma'})$ denotes the covariance between $X_\gamma$ and $X_{\gamma'}$ defined by 
\[
\Cov(X_\gamma, X_{\gamma'}) := \E\Big((X_{\gamma} - \E (X_{\gamma} )) (\overline{X}_{\gamma'} - \E (\overline{X}_{\gamma'} ))\Big).
\]

The Bochner Theorem for positive definite functions on locally compact groups implies that given any weakly $\Gamma$-stationary stochastic process $X = (X_\gamma)_{\gamma\in \Gamma}$,  there exists a unique measure $\mu_X$ on $\O_F$, called the spectral measure of $X$,  such that  
\begin{align}\label{def-sp}
\int_{\O_F}   \gamma(x) \overline{\gamma'(x)} \mu_X (dx)  = \Cov(X_\gamma, X_{\gamma'}). 
\end{align}
\begin{defn}
A  weakly $\Gamma$-stationary stochastic process  $X = (X_\gamma)_{\gamma\in\Gamma}$ defined on a probability space $(\Omega, \PP)$ is called  Kolmogorov minimal, if 
 \begin{align}\label{minimal-rel}
 X_0 - \E(X_0)  \in \overline{\spann}^{L^2(\Omega, \PP)} \Big\{X_\gamma - \E(X_\gamma): \gamma \in \Gamma \setminus \{0\}\Big\}. 
 \end{align}
\end{defn}

\begin{prop}\label{prop-min}
Let $ X =(X_\gamma)_{\gamma\in  \Gamma }$ be a $\Gamma$-stationary stochastic process. Assume that 
\begin{itemize}
\item[(1)] $ \sum_{\gamma\in \Gamma}\Cov(X_\gamma, X_0) =0$;
\item[(2)] $ \sum_{|\gamma|> q^{n}}  | \Cov(X_\gamma, X_0)| = O (q^{-n})$, as $n\to \infty$.
\end{itemize}
Then  $X$ is Kolmogorov minimal. 
\end{prop}

The proof of Proposition \ref{prop-min} is based on  the following Lemma \ref{lem-Kol} due to Kolmogorov. The proof of Lemma \ref{lem-Kol} is similar to  the proof of the same result for $\Z^d$-indexed stochastic processes and the reader is referred to \cite[Lemma 2.1]{QB4}. 
\begin{lem}[The Kolmogorov Criterion]\label{lem-Kol}
 Let  $X = (X_\gamma)_{\gamma\in\Gamma}$ be a weakly $\Gamma$-stationary stochastic process such that $\E(X_\gamma) =\E(X_0)= 0$.   Assume that  the spectral measure $\mu_X$ of $X$ has the Lebesgue decomposition: 
  \[
 \mu_X = \mu_X^a + \mu_X^s,
 \]
 where $\mu_X^a$ and $\mu_X^s$ are the absolutely continuous and the singular parts of $\mu_X$ respectively with respect to $\mathfrak{m}$.  Then the least $L^2(\Omega, \PP)$-distance between $X_0$ and the space  $\overline{\spann}^{L^2} \{X_\gamma: \gamma \in \Gamma \setminus \{0\}\}$ is given by 
\[
\mathrm{dist}(X_0,   \overline{\spann}^{L^2} \{X_\gamma: \gamma \in \Gamma \setminus \{0\}\} ) =   \left(  \int_{\O_F}  \Big( \frac{d \mu_X^a}{d\mathfrak{m}} (x) \Big)^{-1}  d \mathfrak{m} (x)  \right)^{-1/2},
\]
where  the right side is to be interpreted as zero if 
 \begin{align*}
 \int_{\O_F}  \Big( \frac{d \mu_X^a}{d\mathfrak{m}} (x) \Big)^{-1}  d \mathfrak{m} (x) = \infty. 
 \end{align*}
\end{lem}

\begin{lem}\label{lem-zyg}
Let $(c_\gamma)_{\gamma\in \Gamma }$ be a $\Gamma$-indexed sequence of complex numbers such that 
$$
\sum_{| \gamma|  > q^{n}} | c_\gamma| =  O (q^{-n}), \text{\, as $n \to \infty$.}
$$
Then there exists a constant $C>0$, such that the function $f: \O_F \rightarrow \C$ defined by the formula 
$$
f(x) =  \sum_{\gamma\in \Gamma} c_\gamma \gamma(x)
$$
satisfies 
$$
| f(x+h)- f(x) | \le C|h|, \text{\, for any $x, h \in \O_F$.}
$$
\end{lem}

\begin{proof}
Let $h\in \O_F$ be such that $|h| = q^{-n}$. Notice that if $| \gamma| \le q^{n}$, then $\gamma(h) = 1$.  Therefore,  we have
\begin{align*}
f(x+h)- f(x) = \sum_{\gamma \in \Gamma} c_\gamma ( \gamma(h) - 1) \gamma(x) = \sum_{\gamma \in \Gamma, | \gamma|> q^{n}} c_\gamma ( \gamma(h) - 1) \gamma(x). 
\end{align*}
It follows that 
\begin{align*}
| f(x+h)- f(x) |  \le  \sum_{| \gamma|> q^{n}} c_\gamma | \gamma(h) - 1|  \le 2 \sum_{| \gamma|> q^{n}} c_\gamma  =   O(q^{-n}) = O(|h|).
\end{align*}
\end{proof}

\begin{lem}\label{lem-div}
$\displaystyle{\int_{\O_F}} |h|^{-1} \mathrm{d}\mathfrak{m}(h) = \infty$. 
\end{lem}

\begin{proof}
Denote  $S(0, q^{-n}): = \varpi^n \O_F \setminus \varpi^{n+1} \O_F$, then we have the partition
\[
\O_F =  \bigsqcup_{n = 0}^\infty  S(0, q^{-n}).
\]
Note that 
\[
\mathfrak{m}(S(0, q^{-n}))= \mathfrak{m}(\varpi^n \O_F) -  \mathfrak{m}(\varpi^{n+1} \O_F)= q^{-n-1}(q-1).
\]
Since any  $ h \in S(0, q^{-n})$ has the absolute value $|h| = q^{-n}$, we have
\begin{align*}
\int_{\O_F} |h|^{-1} \mathrm{d}\mathfrak{m}(h)   = \sum_{n = 0}^\infty  q^{n} \mathfrak{m}(S(0, q^{-n})) = \infty.
\end{align*}
\end{proof}

\begin{proof}[Proof of Proposition \ref{prop-min}]
The spectral measure of $X = (X_\gamma)_{\gamma \in \Gamma}$ is given by 
\[
\mu_X (dx) =\sum_{\gamma\in \Gamma}\Cov(X_\gamma, X_0) \gamma(x)  \mathfrak{m}(dx).
\]
Indeed, one can easily check that the measure defined on the right hand side satisfies the equation \eqref{def-sp}. By the uniqueness of the spectral measure, it must be $\mu_X$.  Note that here,  the spectral measure $\mu_X$ has no singular part with respect to $\mathfrak{m}$. 

Now let 
\[
f(x)  : = \frac{d \mu_X}{d\mathfrak{m}}(x)  =\sum_{\gamma\in \Gamma}\Cov(X_\gamma, X_0) \gamma(x). 
\]
The first condition in Proposition \ref{prop-min} implies that $f(0) =  0$. While the second condition, combined with Lemma \ref{lem-zyg},  implies that there exists $C>0$, such that
$$
f(h)  = | f(h) - f(0)| \le  C | h|.
$$
It follows that 
\[
\int_{\O_F} f(h)^{-1} \mathrm{d}\mathfrak{m}(h) \ge C^{-1}\int_{\O_F} |h|^{-1} \mathrm{d}\mathfrak{m}(h) =  \infty.
\]
This  combined with Lemma \ref{lem-Kol} implies the desired relation \eqref{minimal-rel}.
\end{proof}

\section{Determinantal point processes on $F$}
\subsection{Proofs of Theorem \ref{main-thm} and Theorem \ref{main-thm-bis}}
 For any positive integer $m > 0$, fix a set of representatives $\mathbb{L}_m$ of  the quotient additive group $F /\varpi^{-m}\O_F$. The cosets of $\varpi^{-m}\O_F$ are exactly the closed balls $B(x, q^m)$ with $x \in \mathbb{L}_m$ and we have
\[
F = \bigsqcup_{x\in \mathbb{L}_m} B(x, q^m).
\]
In what follows, we assume that $\mathbb{L}_m$ contains the origin $0\in F$.  We will also identify  $\mathbb{L}_m$ with $F /\varpi^{-m}\O_F$  and equip  $\mathbb{L}_m$ with the same additive group structure as $F /\varpi^{-m}\O_F$. 

Note that we have a natural group isomorphism between two additive groups:
\begin{align}\label{gp-iso}
\begin{array}{ccc}
\mathbb{L}_m = F/\varpi^{-m} \O_F & \longrightarrow & \Gamma = F/\varpi\O_F
\\
x + \varpi^{-m} \O_F &  \mapsto &  \varpi^{m+1} x + \varpi \O_F
\end{array}.
\end{align}
Using the group isomorphism \eqref{gp-iso},  we immediately get a corollary from Proposition \ref{prop-min}.  

\begin{cor}\label{cor-min}
Fix any positive integer $m > 0$. Let $ X =(X_x )_{x \in  \mathbb{L}_m}$ be an $\mathbb{L}_m$-stationary stochastic process defined on a probability space $(\Omega, \PP)$. Assume that 
\begin{itemize}
\item[(1)] $ \sum_{x \in \mathbb{L}_m }\Cov(X_x, X_0) =0$;
\item[(2)] $ \sum_{x\in \mathbb{L}_m: |x |> q^{n}}  | \Cov(X_x, X_0)| = O (q^{-n})$, as $n\to \infty$.
\end{itemize}
Then  $X$ is Kolmogorov minimal, that is, 
\begin{align*}
 X_0  - \E(X_0) \in \overline{\spann}^{L^2(\Omega, \PP)} \Big\{X_x  -\E(X_x) :  x \in \mathbb{L}_m \setminus \{0\}\Big\}. 
 \end{align*}
\end{cor}

Let $S\subset F$ be a measurable subset such that $ 0 < \mathfrak{m}(S) < \infty$. Let 
\[
K_S(x, y) = \widehat{\1_{S}}(x-y), \text{\, for $x, y \in F.$}
\]
Then the kernel $K_S$ induces a determinantal point process on $F$. In what follows, let $\mathcal{X}_S$ be a determinantal point process induced by the correlation kernel $K_S(x, y)$.  The translation invariance of the kernel 
\[
K_S(x, y) = K_S(x +z, y+z)
\]
implies that the determinantal point process $\X_S$ is translation-invariant.

For each $x \in \mathbb{L}_m$, we set a random variable
\begin{align}\label{n-x}
N_x^{(m)}: = \#(\mathcal{X}_S \cap B(x, q^m)). 
\end{align}
Since the law of  $\X_S$ is invariant under the  translations $x \in \mathbb{L}_m$, the stochastic process $(N_x^{(m)})_{x\in \mathbb{L}_m}$ is also $\mathbb{L}_m$-stationary. 

For simplifying the notation, in what follows, when $m$ is clear from the context,  we will denote $N_x^{(m)}$ by $N_x$.

\begin{lem}\label{lem-vanish}
For any measurable subset $S \subset F$ such that $0 <\mathfrak{m}(S) < \infty$,  we have 
\begin{align}\label{sum-0}
\sum_{x \in \mathbb{L}_m} \Cov(N_x, N_0) =0.
\end{align}
\end{lem}

\begin{proof}
By taking $n=1$ in the formula \eqref{def-DPP}, for any $x\in\mathbb{L}_m$, we have 
\[
\E(N_x) = \E(N_0)= \E\Big(\sum_{u \in \mathcal{X}_S} \1_{B(0, q^{m})} (u)\Big)=\int_{B(0, q^{m})}  \widehat{\1_{S}}(0) \mathrm{d}\mathfrak{m}(z) = q^m \mathfrak{m}(S). 
\]
We can write
\begin{align*}
\E(N_0^2)& = \E\Big(\sum_{u, v  \in \mathcal{X}_S} \1_{B(0, q^{m})} (u) \1_{B(0, q^{m})} (v) \Big)
\\
& =\underbrace{ \E\Big(\sum_{u \in \mathcal{X}_S} \1_{B(0, q^{m})} (u)  \Big)}_{ \text{denoted by $I$} } + \underbrace{\E\Big(\sum_{u, v \in \mathcal{X}_S, \, u \ne v} \1_{B(0, q^{m})} (u) \1_{B(0, q^{m})} (v) \Big)}_{\text{denoted by $II$}}. 
\end{align*}
The first term $I$ has already been shown to be $q^m \mathfrak{m}(S)$. To compute the second term, we take  $n=2$ in \eqref{def-DPP} and get 
\begin{align*}
\\
II &  =  \iint_{B(0, q^{m})^2}  (\widehat{\1_{S}}(0)^2 - |\widehat{\1_{S}}(z_1-z_2) |^2) \mathrm{d}\mathfrak{m}(z_1)\mathrm{d}\mathfrak{m}(z_2)
\\
& =  q^{2m} \mathfrak{m}(S)^2  - \iint_{B(0, q^{m})^2}  |\widehat{\1_{S}}(z_1-z_2) |^2 \mathrm{d}\mathfrak{m}(z_1)\mathrm{d}\mathfrak{m}(z_2).
\end{align*}
Hence 
\begin{align*}
\Cov(N_0, N_0)  & =\Var(N_0, N_0) = \E(N_0^2) - \E(N_0)^2
\\ 
& = q^m \mathfrak{m}(S)   - \iint_{B(0, q^{m})^2}  |\widehat{\1_{S}}(z_1-z_2) |^2 \mathrm{d}\mathfrak{m}(z_1)\mathrm{d}\mathfrak{m}(z_2).
\end{align*}

Note that for $x \in \mathbb{L}_m\setminus\{ 0\}$, the two closed balls $B(0, q^{m})$ and $B(x, q^{m})$ are disjoint. Therefore, if $x \in \mathbb{L}_m\setminus\{ 0\}$, we have 
\begin{align*}
\E(N_0N_x)& =\E\Big(\sum_{u, v \in \mathcal{X}_S } \1_{B(0, q^{m})} (u) \1_{B(x, q^{m})} (v) \Big) 
\\
 & = \E\Big(\sum_{u, v \in \mathcal{X}_S, u \ne v } \1_{B(0, q^{m})} (u) \1_{B(x, q^{m})} (v) \Big). 
\end{align*}
By taking $n =2$ in \eqref{def-DPP}, we get 
\begin{align*}
\E(N_0N_x) & =  \iint_{B(0, q^{m}) \times B(x, q^{m}) }  (\widehat{\1_{S}}(0)^2 - |\widehat{\1_{S}}(z_1-z_2) |^2) \mathrm{d}\mathfrak{m}(z_1)\mathrm{d}\mathfrak{m}(z_2)
\\
& =  q^{2m} \mathfrak{m}(S)^2  - \iint_{B(0, q^{m})\times B(x, q^{m}) }  |\widehat{\1_{S}}(z_1-z_2) |^2 \mathrm{d}\mathfrak{m}(z_1)\mathrm{d}\mathfrak{m}(z_2).
\end{align*}
Consequently, 
\begin{align}\label{cov-Qp}
\Cov (N_x, N_0) = - \iint_{B(0, q^{m}) \times    B(x, q^{m})} |\widehat{\1_{S}}(z_1-z_2) |^2 \mathrm{d}\mathfrak{m}(z_1)\mathrm{d}\mathfrak{m}(z_2).
\end{align}
Finally, by using the equality
\begin{align*}
& \sum_{ x\in  \mathbb{L}_m}\iint_{B(0, q^{m}) \times    B(x, q^{m})} |\widehat{\1_{S}}(z_1-z_2) |^2 \mathrm{d}\mathfrak{m}(z_1)\mathrm{d}\mathfrak{m}(z_2)
\\
& = \iint_{B(0, q^{m}) \times F} |\widehat{\1_{S}}(z_1-z_2) |^2 \mathrm{d}\mathfrak{m}(z_1)\mathrm{d}\mathfrak{m}(z_2)  
\\
&   =  \mathfrak{m}(B(0, q^{m})) \int_{F} |\1_{S}|^2 \mathrm{d}\mathfrak{m}  = q^m \mathfrak{m}(S),
\end{align*}
we obtain the desired equality \eqref{sum-0}. 
\end{proof}

\begin{lem}\label{lem-av-decay}
Let $n$ be an integer such that $n> m$.  We have 
\begin{align}\label{av-decay}
\sum_{x \in \mathbb{L}_m; \, |x |> q^{n}}  | \Cov(N_x, N_0)|  = q^m\int_{F\setminus B(0, q^{n})} |\widehat{\1_{S}} |^2 \mathrm{d}\mathfrak{m}.
\end{align}
\end{lem}
\begin{proof}
Let $x \in \mathbb{L}_m$. Then for any $z \in B(0, q^{m})$, we have 
\[
z + B(x, q^{m})  = B(x, q^m).
\]
Since $\mathfrak{m}$ is the Haar measure on $F$, the restriction $\mathfrak{m}|_{B(x, q^m)}$ is invariant under the translation of $z$. Therefore, the equality \eqref{cov-Qp} can be written as
 \begin{align}\label{f-red}
\Cov (N_x, N_0)  = - q^m\int_{B(x, q^{m})} |\widehat{\1_{S}}|^2 \mathrm{d}\mathfrak{m}.
\end{align}
It follows that 
\begin{align*}
& \sum_{x \in \mathbb{L}_m; \, |x |> q^{n}}  | \Cov(N_x, N_0)|   =  \sum_{x \in \mathbb{L}_m; \, |x |> q^{n}} q^m\int_{B(x, q^{m})} |\widehat{\1_{S}} |^2 \mathrm{d}\mathfrak{m}
\\ 
 &=  q^m\int_{F} |\widehat{\1_{S}} |^2 \mathrm{d}\mathfrak{m} -  \sum_{x  \in \mathbb{L}_m; \,  |x|\le q^{n}} q^m\int_{B(x, q^{m})} |\widehat{\1_{S}} |^2 \mathrm{d}\mathfrak{m}. 
 \end{align*}
By applying the partition 
 \begin{align*}
\bigsqcup_{x \in \mathbb{L}_m; \,  |x|\le q^{n}}B(x, q^{m})= B(0, q^n).
\end{align*}
we obtain the desired equality
\begin{align*}
 \sum_{x \in \mathbb{L}_m; \, |x |> q^{n}}  | \Cov(N_x, N_0)| &  =  q^m\int_{F} |\widehat{\1_{S}} |^2 \mathrm{d}\mathfrak{m} -   q^m\int_{B(0, q^{n})} |\widehat{\1_{S}} |^2 \mathrm{d}\mathfrak{m}
\\
& = q^m\int_{F\setminus B(0, q^{n})} |\widehat{\1_{S}} |^2 \mathrm{d}\mathfrak{m}.
 \end{align*}
\end{proof}

\begin{lem}\label{lem-ref}
Let $n$ be an integer such that $n> m$.  We have 
\begin{align}\label{cont-analogue}
\sum_{x \in \mathbb{L}_m; \, |x |> q^{n}}  | \Cov(N_x, N_0)|  = q^m \dashint_{B(0, q^{-n})}  \mathfrak{m}(   S\setminus  (S+y))\mathrm{d}\mathfrak{m}(y).
\end{align}
\end{lem}

\begin{proof}
Note that we have (see, e.g. \cite[Lemma 3.3]{BQ-sym})
 \[
\1_{B(0, q^{n})}   =q^n  \widehat{\1_{B(0, q^{-n})}}
\]
and also 
\[
 \widehat{\1_{S}} \cdot  \widehat{\1_{B(0, q^{-n})}}  =   \mathcal{F}(\1_{S} * \1_{B(0, q^{-n})}). 
 \]
Here $\mathcal{F}$ denotes also the Fourier transform.  
Therefore, 
 \begin{align*}
  \int_{B(0, q^{n})} |\widehat{\1_{S}} |^2 \mathrm{d}\mathfrak{m} =& q^{2n} \int_{F} |\widehat{\1_{S}} \cdot  \widehat{\1_{B(0, q^{-n})}}  |^2 \mathrm{d}\mathfrak{m}
 \\
  =& q^{2n} \int_{F} |  \mathcal{F}(  \1_{S} * \1_{B(0, q^{-n})})  |^2 \mathrm{d}\mathfrak{m}. 
  \end{align*}
By Parseval's identity 
  \begin{align*}
& \int_{B(0, q^{n})} |\widehat{\1_{S}} |^2 \mathrm{d}\mathfrak{m} =  q^{2n}\int_{F} |    \1_{S} *  \1_{B(0, q^{-n})}   |^2 \mathrm{d}\mathfrak{m}
 \\
 =& q^{2n}\int_{F} \Bigg[ \int_{B(0, q^{-n})}  \1_S(x-y) \mathrm{d}\mathfrak{m}(y)  \int_{B(0, q^{-n})}  \1_S(x-y') \mathrm{d}\mathfrak{m}(y')\Bigg]\mathrm{d}\mathfrak{m}(x) 
 \\
  = & q^{2n} \int_{B(0, q^{-n})}   \int_{B(0, q^{-n})}  \mathfrak{m}(( S+y) \cap (S+y'))\mathrm{d}\mathfrak{m}(y)\mathrm{d}\mathfrak{m}(y'). 
  \end{align*}
Note that  $B(0, q^{-n})$ is an additive group and for any $y, y' \in B(0, q^{-n})$,  we have 
\[
 \mathfrak{m}(( S+y) \cap (S+y')) =  \mathfrak{m}(( S+(y - y')) \cap S). 
\]
Since the restriction $\mathfrak{m}|_{B(0, q^{-n})}$ is invariant under the translations of $y'\in B(0, q^{-n})$, we have  
  \begin{align*}
 \int_{B(0, q^{n})} |\widehat{\1_{S}} |^2 \mathrm{d}\mathfrak{m}
  =&  q^{2n}\mathfrak{m}(B(0, q^{-n})) \int_{B(0, q^{-n})}  \mathfrak{m}(( S+y) \cap S)\mathrm{d}\mathfrak{m}(y) 
 \\
  =& \dashint_{B(0, q^{-n})}  \mathfrak{m}(( S+y) \cap S)\mathrm{d}\mathfrak{m}(y).
 \end{align*}
The  equality  \eqref{av-decay} combined with the following equality
\[
\int_{F} |\widehat{\1_{S}} |^2 \mathrm{d}\mathfrak{m}  = \mathfrak{m}(S) = \dashint_{B(0, q^{-n})}  \mathfrak{m}(S) \mathrm{d}\mathfrak{m}(y)
\]
yields the desired equality 
 \begin{align*}
  & \sum_{x \in \mathbb{L}_m; \, |x |> q^{n}}  | \Cov(N_x, N_0)|   = q^m\int_{F\setminus B(0, q^{n})} |\widehat{\1_{S}} |^2 \mathrm{d}\mathfrak{m} 
 \\
 &   =  q^{m} \dashint_{B(0, q^{-n})}  \Big(\mathfrak{m}(S) - \mathfrak{m}(( S+y) \cap S)\Big)\mathrm{d}\mathfrak{m}(y)
 \\
 & = q^{m} \dashint_{B(0, q^{-n})}  \mathfrak{m}(S \setminus (S + y) ) \mathrm{d}\mathfrak{m}(y). 
 \end{align*}
\end{proof}

\begin{proof}[Proof of Theorem \ref{main-thm}]
In Definition \ref{def-rigid}, the subset $B$ ranges over all bounded open subsets.  It is easy to see that without changing the definition of number rigidity, we can let $B$  only range over all the closed balls $B(0, q^{m})$ with $m\in \N$.  In our notation \eqref{n-x},  for each $m \in \N$, the number of particles of our determinantal point process $\X_S$ inside $B(0, q^m)$ is denoted by $N_0^{(m)}$, by Corollary \ref{cor-min},  Lemma \ref{lem-vanish} and Lemma \ref{lem-ref}, the assumption \eqref{con-geo} implies 
\begin{align*}
 N_0^{(m)}  - \E(N_0^{(m)}) \in \overline{\spann}^{L^2(\Omega, \PP)} \Big\{N_x^{(m)}  -\E(N_x^{(m)}) :  x \in \mathbb{L}_m \setminus \{0\}\Big\}. 
 \end{align*}
 Thus the random variable  $N_0^{(m)}$ is measurable with respect to the completion of the $\sigma$-algebra generated by  the family $\{N_x^{(m)}:  x \in \mathbb{L}_m \setminus \{0\}\}$. Therefore,  $N_0^{(m)}$ is almost surely determined by $\X_S \setminus B(0, q^m)$.  Since $m$ is arbitrary,  we complete the proof of the number rigidity of $\X_S$. 
\end{proof}

\begin{proof}[Proof of Theorem \ref{main-thm-bis}]
In the proof of Theorem \ref{main-thm}, replacing Lemma \ref{lem-ref} by Lemma \ref{lem-av-decay}, we obtain Theorem \ref{main-thm-bis}. 
\end{proof}

\subsection{Examples}\label{sec-ex}
Let us now concentrate in the case where $S \subset F$ is an open subset with $0 < \mathfrak{m}(S) < \infty$. 
It is easy to see that  $S$ has a unique decomposition 
\begin{align}\label{dec-S}
S= \bigsqcup_{k=1}^N B(x_k, q^{-n_k}), 
\end{align}
such that $B(x_k, q^{n_k})$ (with $n_k \in \Z$) is the {\bf largest} closed ball in $S$ containing $x_k \in S$. Here $N \in \N \cup \{\infty\}$.  

In the decomposition \eqref{dec-S} of $S$, we may assume that 
\[
l_0(S) = n_1 \le n_2 \le \cdots,
\]
where $l_0(S)$ is the smallest integer in the sequence $(n_k)_{k=1}^N$. For any $l\ge l_0(S)$, we define  the multiplicity of $l$ in the sequence $(n_k)_{k=1}^N$ by 
\begin{align}\label{mul-def}
m_l(S): = \# \{k: n_k = l\}. 
\end{align}

Fix $n \ge  l_0(S)$. For  any given $|y| =q^{-l}$ with $l \ge n$, it is easy to see that all those sub-balls $B(x_k, q^{n_k})$ in the decomposition \eqref{dec-S} whose radius is not smaller than $q^{-l}$, we have 
\[
B(x_k, q^{n_k}) + y = B(x_k, q^{n_k}). 
\] 
Therefore, by recalling the definition \eqref{mul-def} for the multiplicity, we have 
\begin{align}\label{pt-ctr}
\mathfrak{m}(S\setminus (S+y) )  \le  \sum_{a =  l+1}^\infty  m_a (S) \cdot \mathfrak{m}( B(0, q^{-a})) =  \sum_{a =  l+1}^\infty  m_a(S) q^{-a}. 
\end{align}

\begin{prop}\label{prop-mul}
If $M(S): = \sup_l m_l(S) < \infty$, then the  open subset $S$ satisfies the condition \eqref{con-geo}.
\end{prop}
\begin{proof}
For  any given $|y| =q^{-l}$ with $l \ge l_0(S)$,  the inequality \eqref{pt-ctr} implies that 
\[
\mathfrak{m}(S\setminus (S+y) )  \le    M(S) \sum_{a =  l+1}^\infty  q^{-a} =  M(S)   \frac{q^{-l}}{q-1}.
\]
Therefore, we have
\begin{align*}
& \dashint_{B(0, q^{-n})}\mathfrak{m}(S\setminus (S+y) ) \mathrm{d}\mathfrak{m}(y)  
\\
&  \le  q^n \sum_{l=n}^\infty \int_{\{y \in F: | y|= q^{-l}\}: }\mathfrak{m}(S\setminus (S+y) ) \mathrm{d}\mathfrak{m}(y)  
\\
& \le  q^n M(S) \sum_{l=n}^\infty   \frac{q^{-l}}{q-1}  \mathfrak{m}(\{y \in F: | y|= q^{-l}\})
\\
& =  q^n  M(S) \sum_{l=n}^\infty  q^{-2 l +1} = M(S) \frac{q^{3}}{q^2-1} q^{-n}.
\end{align*}
This implies that $S$ satisfies the  desired condition \eqref{con-geo}.
\end{proof}

\section*{Acknowlegements}
The author  is supported by the grant IDEX UNITI - ANR-11-IDEX-0002-02, financed by Programme ``Investissements d'Avenir'' of the Government of the French Republic managed by the French National Research Agency. During an earlier stage of this research, he was also partially supported by the  grant  346300 for IMPAN from the Simons Foundation and the matching 2015-2019 Polish MNiSW fund.

%

\begin{thebibliography}{10}

\bibitem{Buf-rigid}
Alexander~I. Bufetov.
\newblock Rigidity of determinantal point processes with the {A}iry, the
  {B}essel and the gamma kernel.
\newblock {\em Bull. Math. Sci.}, 6(1):163--172, 2016.


\bibitem{QB4}
 Alexander~I. Bufetov,  Yoann~Dabrowski and Yanqi Qiu.
\newblock Linear rigidity of stationary stochastic processes.
\newblock {\em arXiv:1507.00670, to appear in  Ergodic Theory Dynam. Systems}.



\bibitem{QB3}
Alexander~I. Bufetov and Yanqi Qiu.
\newblock Determinantal point processes associated with {H}ilbert spaces of
  holomorphic functions.
\newblock {\em arXiv:1411.4951, to appear in Commun. Math. Phys.}



\bibitem{BQ-sym}
Alexander~I. Bufetov and Yanqi Qiu.
\newblock {E}rgodic measures on spaces of infinite matrices over
  non-{A}rchimedean locally compact fields.
\newblock {\em arXiv:1605.09600}.


\bibitem{Ghosh-rigid3}
Subhro Ghosh and Manjunath Krishnapur.
\newblock Rigidity hierarchy in random point fields: random polynomials and
  determinantal processe.
\newblock {\em arXiv:1510.08814}.

\bibitem{Ghosh-sine}
Subhroshekhar Ghosh.
\newblock Determinantal processes and completeness of random exponentials: the
  critical case.
\newblock {\em Probability Theory and Related Fields}, pages 1--23, 2014.

\bibitem{Ghosh-rigid}
Subhroshekhar Ghosh and Yuval Peres.
\newblock Rigidity and tolerance in point processes: {G}aussian zeros and
  {G}inibre eigenvalues,.
\newblock {\em arXiv:1211.3506, to appear in Duke Math. J.}

\bibitem{DPP-HKPV}
J.~Ben Hough, Manjunath Krishnapur, Yuval Peres, and B{\'a}lint Vir{\'a}g.
\newblock Determinantal processes and independence.
\newblock {\em Probab. Surv.}, 3:206--229, 2006.

\bibitem{DPP-L}
Russell Lyons.
\newblock Determinantal probability measures.
\newblock {\em Publ. Math. Inst. Hautes \'Etudes Sci.}, (98):167--212, 2003.

\bibitem{DPP-M}
Odile Macchi.
\newblock The coincidence approach to stochastic point processes.
\newblock {\em Advances in Appl. Probability}, 7:83--122, 1975.

\bibitem{Osada-Shirai}
Hirofumi Osada and Tomoyuki Shirai.
\newblock Absolute continuity and singularity of {P}alm measures of the
  {G}inibre point process.
\newblock {\em Probab. Theory Related Fields}, 165(3-4):725--770, 2016.

\bibitem{DR-fourier}
Dinakar Ramakrishnan and Robert~J. Valenza.
\newblock {\em Fourier analysis on number fields}, volume 186 of {\em Graduate
  Texts in Mathematics}.
\newblock Springer-Verlag, New York, 1999.

\bibitem{DPP-S}
Alexander Soshnikov.
\newblock Determinantal random point fields.
\newblock {\em Uspekhi Mat. Nauk}, 55(5(335)):107--160, 2000.

\end{thebibliography}
%

\def\cprime{$'$} \def\cydot{\leavevmode\raise.4ex\hbox{.}} \def\cprime{$'$}

\end{document}